\documentclass[12pt]{amsart}

 \date{5 February 2015}

 \title{Forbidden subgraphs in the norm graph}
\author[Ball]{Simeon Ball }\thanks{The first author acknowledges the support of the project MTM2008-06620-C03-01 of the Spanish Ministry of Science and Education and the project 2014-SGR-1147 of the Catalan Research Council.}
\author[Pepe]{Valentina Pepe}\thanks{The second author acknowledges the support of the project "Decomposizione, propriet\`{a} estremali di grafi e combinatoria di polinomi ortogonali" of the SBAI Department of Sapienza University of Rome.}
\usepackage{amsmath}
\usepackage{amssymb}
\usepackage{amsthm}
\usepackage{graphicx}
\usepackage{amscd}

 \newtheorem{theorem}{\sc Theorem}[section]

 \makeatletter
 \@addtoreset{equation}{section}
 \makeatother

 \makeatletter
 \@addtoreset{table}{section}
 \makeatother

\begin{document}

\begin{abstract}
We show that the norm graph constructed in \cite{KRS1996} with $n$ vertices about $\frac{1}{2}n^{2-1/t}$ edges, which contains no copy of $K_{t,(t-1)!+1}$, does not contain a copy of $K_{t+1,(t-1)!-1}$.
\end{abstract}

\maketitle

\section{Introduction}

Let $H$ be a fixed graph. The {\em Tur\'an number} of $H$, denoted $ex(n,H)$, is the maximum number of edges a graph with $n$ vertices can have, which contains no copy of $H$. The Erd{\H o}s-Stone theorem from \cite{ES1946} gives an asymptotic formula for the Tur\'an number of any non-bipartite graph, and this formula depends on the chromatic number of the graph $H$.

When $H$ is a complete bipartite graph, determining the Tur\'an number is related to the ``Zarankiewicz problem'' (see \cite{B}, Chap. VI, Sect.2, and \cite{Furedi1996b} for more details and
references).
In many cases even the question of determining the right order of magnitude for $ex(n,H)$ is not known.

Let $K_{t,s}$ denote the complete bipartite graph with $t$ vertices in one class and $s$ vertices in the other. K{\H o}vari, S{\'o}s and Tur\'an \cite{KST1954} proved that for $s \geqslant t$
\begin{equation} \label{upper}
ex(n,K_{t,s}) \leqslant \genfrac{}{}{}{1}{1}{2}(s-1)^{1/t}n^{2-1/t}+\genfrac{}{}{}{1}{1}{2}(t-1)n.
\end{equation}

The norm graph $\Gamma (t)$, which we will define the next section, has $n$ vertices and about $\frac{1}{2}n^{2-1/t}$ edges. In \cite{ARS1999} (based on results from \cite{KRS1996}) it was proven that the graph $\Gamma (t)$ contains no copy of $K_{t,(t-1)!+1}$, thus proving that for $s \geqslant (t-1)!+1$,
$$
ex(n,K_{t,s}) >c n^{2-1/t}
$$
for some constant $c$.

In \cite{BP2012}, it was shown that $\Gamma (4)$ contains no copy of $K_{5,5}$, which improves on the probabilistic lower bound of Erd\H os and Spencer \cite{ES1974} for $ex(n,K_{5,5})$. In this article, we will generalise this result and prove that $\Gamma (t)$ contains no copy of $K_{t+1,(t-1)!-1}$. For $t \geqslant 5$, this does not improve the probabilistic lower bound of Erd\H os and Spencer,
$$
ex(n,K_{t,s}) \geqslant cn^{2-(s+t-2)/(st-1)}.
$$
As far as we are aware, it is however the deterministic construction of a graph with $n$ vertices containing no $K_{t+1,(t-1)!-1}$ with the most edges.

\section{The norm graph}

Suppose that $q=p^h$, where $p$ is a prime, and denote by ${\mathbb F}_q$ the finite field with $q$ elements. We will use the following properties of finite fields. For any $a,b \in {\mathbb F}_q$, $(a+b)^{p^i}=a^{p^i}+b^{p^i}$, for any $i \in {\mathbb N}$. Note that $(a-b)^{p^i}=a^{p^i}-b^{p^i}$, since either $p^i$ is odd or $-1=1$. Secondly, for all $a \in {\mathbb F}_{q^i}$, $a^q=a$ if and only if $a \in {\mathbb F}_q$. Finally $N(a)=a^{1+q+\cdots + q^{k-1}} \in {\mathbb F}_q$, for all $a \in {\mathbb F}_{q^k}$, since $N(a)^{q}=N(a)$.

Let $\mathbb{F}$ denote an arbitrary field.
We denote by $\mathbb{P}_n(\mathbb{F})$ the projective space arising from the $(n+1)$-dimensional vector space over $\mathbb{F}$. Throughout $\dim$ will refer to projective dimension. A point of
$\mathbb{P}_n(\mathbb{F})$ (which is a one-dimensional subspace of the vector space) will often be written as $\langle u \rangle$, where $u$ is a vector in the $(n+1)$-dimensional vector space over $\mathbb{F}$.

Let $\Gamma (t)$ be the graph with vertices $(a,\alpha) \in {\mathbb F}_{q^{t-1}} \times {\mathbb F}_q$, $\alpha \neq 0$, where $(a,\alpha)$ is joined to $(a',\alpha')$ if and only if $N(a+a')=\alpha\alpha'$. The graph $\Gamma (t)$ was constructed in \cite{KRS1996}, where it was shown to contain no copy of $K_{t,t!+1}$. In \cite{ARS1999} Alon, R\'onyai and Szab\'o proved that $\Gamma (t)$ contains no copy of $K_{t,(t-1)!+1}$.  Our aim here is to show that it also contains no $K_{t+1,(t-1)!-1}$, generalizing the same result for $t=5$ presented in \cite{BP2012}.

Let
$$
V=\{ (1,a)\otimes (1,a^q) \otimes \cdots \otimes (1,a^{q^{t-2}}) \ | \ a \in {\mathbb F}_{q^{t-1}} \} \subset \mathbb{P}_{2^{t-1}-1}(\mathbb{F}_{q^{t-1}}).
$$

The set $V$ is the affine part of an algebraic variety that is in turn a subvariety of the Segre variety $$
\Sigma=\underbrace{\mathbb{P}_1 \times \mathbb{P}_1 \times \cdots \times \mathbb{P}_1}_{t-1 \text{ times}},
$$
where $\mathbb{P}_1=\mathbb{P}_1({\mathbb F}_q)$.

The affine point $(1,a)\otimes (1,a^q) \otimes \cdots \otimes (1,a^{q^{t-2}})$ has coordinates indexed by the subsets of $T:=\{0,1,\ldots,t-1\}$, where the $S$-coordinate is
$$
(\prod_{i \in S}a^{q^i}),
$$
for any non-empty subset $S$ of $T$ and
$$
\displaystyle\prod_{i \in S}a^{q^i}=1
$$
when $S=\emptyset$ (see \cite{Pepe2011}).

Let $n=2^{t-1}-1$.

We order the coordinates of $\mathbb{P}_{n}(\mathbb{F}_{q^{t-1}})$ so that if the $i$-th coordinate corresponds to the subset $S$, then the $(n-i)$--th coordinate corresponds to the subset $T\setminus S$.

Embed the $\mathbb{P}_{n}(\mathbb{F}_{q^{t-1}})$ containing $V$ as a hyperplane section of $\mathbb{P}_{n+1}(\mathbb{F}_{q^{t-1}})$ defined by the equation $x_{n+1}=0$.

Let $\beta$ be the symmetric bilinear form on the $(n+2)$-dimensional vector space over $\mathbb{F}_{q^{t-1}}$ defined by
$$
\beta(u,v)=\sum_{i=0}^{n}u_i v_{n-i}- u_{n+1}v_{n+1}.
$$
Let $\perp$ be defined in the usual way, so that given a subspace $\Pi$ of $\mathbb{P}_{n+1}(\mathbb{F}_{q^{t-1}})$, $\Pi^{\perp}$ is the subspace of $\mathbb{P}_{n+1}(\mathbb{F}_{q^{t-1}})$ defined by
$$
\Pi^{\perp}=\{ v \ | \ \beta(u,v)=0, \ \mathrm{for} \  \mathrm{all} \  u \in \Pi \}.
$$

We wish to define the same graph $\Gamma (t)$, so that adjacency is given by the bilinear form. Let $P_{\infty}=(0,0,0,\ldots,1)$. Let $\Gamma'$ be a graph with vertex set the set of points on the lines joining the points of $V$ to $P_{\infty}$ obtained using only scalars in $\mathbb{F}_q$, distinct from $P_{\infty}$ and not contained in the hyperplane $x_{n+1}=0$. Join two vertices $\langle u \rangle$ and $\langle u' \rangle$ in $\Gamma'$ if and only if $\beta(u,u')=0$. It is a simple matter to verify that the graph $\Gamma'$ is isomorphic to the graph $\Gamma(t)$ since
$$
N(a+b)=\sum_{S \subseteq T} \prod_{i \in S,\  j \in T\setminus S} a^{q^i}b^{q^j}=\beta(u,v)+u_{n+1}v_{n+1},
$$
where
$$
u=(1,a)\otimes (1,a^q) \otimes \cdots \otimes (1,a^{q^{t-2}}),
$$
and
$$
v=(1,b)\otimes (1,b^q) \otimes \cdots \otimes (1,b^{q^{t-2}}).
$$
We shall refer to $\Gamma'$ as $\Gamma (t)$ from now on.

We recall some known properties of $\Sigma$ and its subvariety
$$\mathcal{V}=\{(a,b)\otimes (a^q,b^q)\otimes \cdots (a^{q^{t-2}},b^{q^{t-2}}) \ | \ (a,b) \in \mathbb{P}_1(\mathbb{F}_{q^{t-1}})\}$$
and prove a new one in Theorem~\ref{main}.

Let $\overline{{\mathbb F}_{q}}$ denote the algebraic closure of ${\mathbb F}_q$ and consider $\Sigma$ as the Segre variety over $\overline{{\mathbb F}_{q}}$.

\begin{theorem} \label{smooth}
$\Sigma$ is a smooth irreducible variety.
\end{theorem}

\begin{theorem}
The dimension of $\Sigma$ (as algebraic variety) is $t-1$ and its degree is $(t-1)!$.
\end{theorem}

\begin{theorem}\cite{Pepe2011}\label{independence1}
Any $t$ points of $\mathcal{V}$ are in general position.
\end{theorem}

\begin{theorem}\cite{GP}\label{independence2}
If $t+1$ points span a $(t-1)$-dimensional projective space, then that space contains $q+1$ points of $\mathcal{V}$.
\end{theorem}

\begin{theorem}\label{main}
If a subspace of codimension $t$ contains a finite number of points of $\Sigma$ then it contains at most $(t-1)!-2$ points of $\Sigma$.
\end{theorem}
\begin{proof}
By Theorem~\ref{smooth}, $\Sigma$ is smooth, so it is regular at each of its points, i.e., if $T_P\Sigma$ is the tangent space of $\Sigma$ at a point $P \in \Sigma$, then $\dim T_P\Sigma=t-1$.

Let $\Pi$ be a subspace of codimension $t$ containing a finite number of points of $\Sigma$. Let $P \in \Pi \cap \Sigma$. Then $\dim \langle T_P\Sigma, \Pi \rangle \leqslant n-1$. Therefore, there is a hyperplane $H$ containing $\langle T_P\Sigma, \Pi \rangle$.

Suppose that $H$ contains another tangent space $T_R\Sigma$, with $R \in \Pi \cap \Sigma$. The algebraic variety $H\cap \Sigma$ has dimension $t-2$ (since $\Sigma$ is irreducible) and it has two singular points, $P$ and $R$. Since $\dim H\cap \Sigma=t-2$ as an algebraic variety, there must be a linear subspace $\Pi_1$ of codimension $t-2$ in $H$ containing $\Pi$ and such that $\Pi_1 \cap H\cap \Sigma$ consists of $\deg (H \cap \Sigma) \leqslant (t-1)!$ points of $\Sigma$ counted with their multiplicity. Since $\Pi_1$ contains $P$ and $R$, which are singular points and so with multiplicity at least $2$, we have that
$$
|\Pi\cap \Sigma|\leqslant |\Pi_1\cap\Sigma|\leqslant (t-1)!-2.
$$

Suppose now that $H$ does not contain any other tangent space $T_R\Sigma$ with $R \in \Pi \cap \Sigma$, $R \neq P$. Then take $R \in \Pi \cap \Sigma$ and consider a hyperplane $H'\neq H$ containing $\langle T_R\Sigma, \Pi \rangle$. Then the tangent spaces of $P$ and $R$ with respect to $H\cap H'\cap \Sigma$ are $T_P\Sigma \cap H'$ and $T_R\Sigma \cap H$, and they both have dimension $t-2$ (as linear spaces).

If $\dim H\cap H'\cap \Sigma=t-3$ as an algebraic variety, then $P$ and $R$ are two singular points of $H\cap H'\cap \Sigma$ and we can find, as before, a linear subspace $\Pi_1$ of codimension $t-3$ in $H\cap H'$ such that it contains $\Pi$ and intersects $H\cap H'\cap \Sigma$ in $\deg (H \cap H'\cap \Sigma) \leqslant (t-1)!$ points, counted with their multiplicity. Since $P$ and $R$ have multiplicity at least $2$, we have
$$
|\Pi\cap \Sigma|\leqslant |\Pi_1\cap\Sigma|\leqslant (t-1)!-2.
$$

If $\dim H\cap H'\cap \Sigma=t-2$ as an algebraic variety, then $H\cap \Sigma$ is reducible. Hence, we have
$$
H\cap \Sigma=\mathcal{V}_1 \cup \mathcal{V}_2 \cup \cdots \cup \mathcal{V}_r,
$$
where $\mathcal{V}_i$ is an irreducible variety of dimension $t-2$, for all $i=1,\ldots,r$. So we have
$$
H\cap H'\cap \Sigma=\mathcal{V}_1 \cup \mathcal{V}_2 \cup \cdots  \cup \mathcal{V}_s \cup \mathcal{W}_{s+1}\cup \mathcal{W}_{s+2}\cup \cdots \cup \mathcal{W}_{r},
$$
where $\mathcal{W}_i$ is a hyperplane section of $\mathcal{V}_i$, for all $i=s+1,\ldots,r$.
We observe that also $H'\cap\Sigma$ has to be reducible and, since the decomposition in irreducible components is unique, we have
$$
H'\cap \Sigma=\mathcal{V}_1 \cup \mathcal{V}_2 \cup \cdots \cup \mathcal{V}_s\cup \mathcal{V}'_{s+1} \cup \mathcal{V}'_{s+2}\cup \cdots \cup \mathcal{V}'_r,$$
where $\mathcal{V}_i$ and $\mathcal{V}'_j$ are irreducible varieties of dimension $t-2$.

We have, by hypothesis, that $T_P\Sigma \subset H$ and $P \in \Pi$. So either $P\in \mathcal{V}_i$ and it is singular for $\mathcal{V}_i$, for some $i \in \{1,2,\ldots,r\}$, or it is not singular for  $\mathcal{V}_{\ell}$, for any $\ell \in \{1,2,\ldots,r\}$.

Suppose we are in the first case. We know that $P \in \Pi \subset H'$. If $\mathcal{V}_i \subseteq H'$, then $P$ is singular for an irreducible component of $H'\cap \Sigma$ and so $T_P\Sigma \subset H'$, contradicting our hypothesis, so $\mathcal{V}_i$ is not contained in $H'$ and $H' \cap \mathcal{V}_i=\mathcal{W}_i$. We have that $\dim T_P\Sigma\cap H'=t-2$ (as linear subspace) and $\dim \mathcal{W}_i=t-3$ (as algebraic variety), so $P$ is singular for $\mathcal{W}_i$.

Suppose now that $P$ is not singular for any $\mathcal{V}_i$, so the dimension of $T_P \mathcal{V}_i$, as a subspace, is $t-2$. If $P \not\in \mathcal{V}_j$, for any $i\neq j$, then
$$
T_P(H \cap \Sigma)=T_P(\mathcal{V}_i)=T_P(\Sigma),
$$
a contradiction since the dimension of $T_P(\Sigma)$ is $t-1$. Hence $P \in \mathcal{V}_i \cap \mathcal{V}_j$, and so $P$ is contained in the intersection of two components of $H'\cap \Sigma$, so it is again a singular (or multiple) point. The same is true for the point $R$ such that $T_R\Sigma \subset H'$, so in
$$
\mathcal{V}_1 \cup \mathcal{V}_2 \cup \cdots  \cup \mathcal{V}_s \cup \mathcal{W}_{s+1}\cup \mathcal{W}_{s+2}\cup \cdots \cup \mathcal{W}_{r}$$
there are at least two multiple points and when we sum up all the degrees, we count at least two points twice, hence, by
$$ \sum_{i=1}^s \deg \mathcal{V}_i+\displaystyle \sum_{j=s+1}^r \deg \mathcal{W}_j \leqslant (t-1)!,
$$
we get that the number of points in
$$\Pi \cap (\mathcal{V}_1 \cup \mathcal{V}_2 \cup \cdots  \cup \mathcal{V}_s \cup \mathcal{W}_{s+1}\cup \mathcal{W}_{s+2}\cup \cdots \cup \mathcal{W}_{r}),$$
is at most $(t-1)!-2$.
\end{proof}

{\bf Remark} One could wonder whether one could try with one more hyperplane $H''$ such that $T_Q\Sigma \subset H''$, $T_Q\Sigma\nsubseteq H$,$T_Q\Sigma\nsubseteq H$' and $Q \in \Pi$. However, it can happen that $H \cap H' \cap H''=H\cap H'$, so $\dim T_Q\Sigma \cap H\cap H'\cap H''=t-2$ (as linear space) and $\dim H\cap H'\cap H''\cap \Sigma=t-2$, so $Q$ would not be a singular point of
$$H\cap H'\cap H''\cap \Sigma=H\cap H'\cap \Sigma.$$

\begin{theorem}
For $q \geqslant (t-1)!+1$ the graph $\Gamma (t)$ contains no $K_{t+1,(t-1)!-1}$.
\end{theorem}

\begin{proof}
Let $X=\{x_1,x_2,\ldots,x_{t+1}\}$ be $t+1$ distinct vertices of $\Gamma (t)$. The set of common neighbours of the elements of $X$ is $\Pi^{\perp} \cap \Gamma(t)$, where $\Pi$ is the subspace spanned by $X$. If any two elements of $X$ project from $P_{\infty}$ onto the same point of $V$, then $P_{\infty} \in \Pi$ and hence $\Pi^{\perp} \subset P_{\infty}^{\perp}$. Since $P_{\infty}^{\perp}$ is the hyperplane $x_{n+1}=0$, $\Pi^{\perp} \cap \Gamma (t)=\emptyset$, and the elements of $X$ have no common neighbour.

Therefore, we assume now that all the points in $X$ project from $P_{\infty}$ onto distinct points of $V$. Then, by Theorem \ref{independence1}, $\dim \Pi\geqslant t-1$.

If $\dim \Pi=t-1$, then by Theorem~\ref{independence1}, the projection of $\Pi$ onto $V$ contains at least $q$ points of $V$. Therefore, there are at least $q$ points $Y$ of $\Pi$ on the lines joining $P_{\infty}$ to the points of $V$. We wish to prove that the points of $Y$ are vertices of the graph $\Gamma (t)$. To do this, we have to show that the points of $Y$, which are of the form $\langle (v,\lambda) \rangle$, where $v \in V$ and $\lambda \in \overline{{\mathbb F}_{q}}$, are of the form $\langle (v,\lambda) \rangle$, where $v \in V$ and $\lambda \in {\mathbb F}_{q}$. Assuming that the vertices in $X$ have at least two common neighbours, we can suppose that there is a common neighbour of the elements of $X$ of the form $\langle (u,\mu) \rangle$, where $u \in V$, $u \neq -v$ and $\mu \in {\mathbb F}_q$, is a common neighbour of the elements of $X$. Then $\langle (u,\mu) \rangle$ is in $\Pi^{\perp}$ and since $Y \subset \Pi$,
$$
N(u+v)=\lambda\mu.
$$
Since $N(u+v) \in {\mathbb F}_q$ and $\mu \in {\mathbb F}_q$, we have that $\lambda \in {\mathbb F}_q$ and so the points of $Y$ are vertices of the graph $\Gamma (t)$. Therefore, the vertices of $X$ have at least $q$ common neighbours.
Since $\Gamma$ contains no $K_{t,(t-1)!+1}$, if $q \geqslant (t-1)!+1$, then this case cannot occur.

If $\dim \Pi=t$ then $\dim \Pi^{\perp}=n-t$. Let $Y$ be the points of $\Pi^{\perp}$ which project from $P_{\infty}$ onto $V$. Arguing as in the previous paragraph, the points $Y$ are vertices of the graph $\Gamma (t)$. Since the vertices of $X$ have at most $(t-1)!+1$ common neighbours, there are a finite number of points in $Y$ and so a finite number of points in the projection of $\Pi^{\perp}$ onto $V$. By Theorem~\ref{main}, this projection contains at most $(t-1)!-2$ points of $V$, so there are at most $(t-1)!-2$ points in $Y$. Therefore, the vertices in $X$ have at most $(t-1)!-2$ common neighbours.
\end{proof}

\bibliographystyle{plain}

{\small Simeon Ball}  \\
{\small Departament de Matem\`atica Aplicada IV}, \\
{\small Universitat Polit\`ecnica de Catalunya, Jordi Girona 1-3},
{\small M\`odul C3, Campus Nord,}\\
{\small 08034 Barcelona, Spain} \\
{\small {\tt simeon@ma4.upc.edu}}

{\small Valentina Pepe} \\
{\small SBAI Department}, \\
{\small Sapienza University of Rome},
{\small Via Antonio Scarpa 16 \\
{\small 00161 Rome, Italy}\\
{\small {\tt valepepe@sbai.uniroma1.it}}

\end{document}